\documentclass{amsart}
\usepackage{amssymb,latexsym}
\usepackage{amsmath}
\usepackage[dvipdfm]{graphicx}
 \usepackage{color}
\usepackage{enumerate}
\newenvironment{enumeratei}{\begin{enumerate}[\upshape (i)]}{\end{enumerate}}
\numberwithin{equation}{section}
\theoremstyle{plain}
 \newtheorem{theorem}{Theorem}[section]
 \newtheorem{lemma}[theorem]{Lemma}
 \newtheorem{proposition}[theorem]{Proposition}
 \newtheorem{corollary}[theorem]{Corollary}
\theoremstyle{definition}
 \newtheorem{definition}[theorem]{Definition}

\theoremstyle{remark}
 \newtheorem*{caseone}{Case $1$}
 \newtheorem*{casetwo}{Case $2$}
 \newtheorem*{casethree}{Case $3$}
\newcommand \url [1] {\tt{#1}}

\newcommand \upcover [1] {#1^\ast}
\newcommand \slft {\lambda} 
\newcommand \eslft {\lambda^{\kern-1pt \mathord=}}
\newcommand \srght {\varrho}
\newcommand \esrght {\varrho^{\kern 0pt \mathord=}}
\newcommand \leqslft {\lambda^{\kern-1pt \mathord\leq}}
\newcommand \leqsrght {\varrho^{\kern-0pt \mathord\leq}}
\newcommand \geqslft {\lambda^{\kern-1pt \mathord\geq}}
\newcommand \geqsrght {\varrho^{\kern-0pt \mathord\geq}}
\newcommand \biggsrght {\varrho^{\kern-0pt \mathord>}}
\newcommand \lessslft {\lambda^{\kern-1pt \mathord<}}
\newcommand \lesssrght {\varrho^{\kern-0pt \mathord<}}
\newcommand \biggslft {\lambda^{\kern-1pt \mathord>}}
\newcommand \lft {\mathrel{\slft}} 
\newcommand \rght {\mathrel{\srght}}
\newcommand \elft {\mathrel{\eslft}}   
\newcommand \leqlft {\mathrel{\leqslft}} 
\newcommand \leqrght {\mathrel{\leqsrght}} 
\newcommand \geqlft {\mathrel{\geqslft}} 
\newcommand \leslft {\mathrel{\lessslft}}
\newcommand \lesrght {\mathrel{\lesssrght}}
\newcommand \bigglft {\mathrel{\biggslft}} 
\newcommand \biggrght {\mathrel{\biggsrght}} 
\newcommand \geqrght {\mathrel{\geqsrght}} 

\newcommand \nnul {\tilde 0}

\newcommand \plu [1] {#1^{\kern 0pt{\mathord {\pmb+}}}}
\newcommand \Elig {E} 
\newcommand\filter[1]{\mathord\uparrow #1}
\newcommand\hcofilter[1]{{\mathord\uparrow}^{\textup{hco}}\kern 0.5pt #1}
\newcommand\phcofilter[2]{{\mathord\uparrow}^{\textup{hco}}_{#1}\kern 0.5pt #2}
\newcommand \Hco [1] {F_{\textup{hco}}(#1)}
\newcommand \dleq {\mathrel{\leq^d}}

\newcommand \dsleq {\mathrel{<^d}}
\newcommand \dsgeq {\mathrel{>^d}}
\newcommand \Max {\textup{Max}}
\newcommand \Min {\textup{Min}}
\newcommand \Nar [1] {\textup{Nar}(#1)}
\newcommand \refl [1] {\textup{VFlip}(#1)}
\newcommand \alg [1] {\mathfrak #1}
\newcommand \lmost [1] {\textup{lbe}(#1)}
\newcommand \rmost [1] {\textup{rbe}(#1)}
\newcommand \eptofilt {\phi}
\newcommand \filttoep {\pi}
\newcommand \Betw [2] {\textup{Betw}(#1,#2)}
\newcommand \MinBetw [2] {\textup{Min}\,\textup{Betw}(#1,#2)}
\newcommand \nonparallel {\mathrel{
\not\mathord{\kern -1.5 pt\parallel}}}

\newcommand \lbound [1] {\textup C_{\textup l}(#1)}
\newcommand \rbound [1] {\textup C_{\textup r}(#1)}
\newcommand \lsp [1] {\textup{lsp}(#1)}
\newcommand \rsp [1] {\textup{rsp}(#1)}
\newcommand \lds [1] {\textup{lds}(#1)}
\newcommand \rds [1] {\textup{rds}(#1)}
\newcommand\ideal[1]{\mathord\downarrow #1}
\newcommand \powset [1] {\textup{PowSet}(#1)}
\newcommand \tuple [1] {\langle #1\rangle}
\newcommand \pair [2] {\tuple{#1,#2}}
\renewcommand\phi{\varphi}
\renewcommand\emptyset{\varnothing}

\newcommand \tbf [1] {\textbf{#1}} 
\newcommand \set[1] {\{#1\}}
\newcommand \bigset[1] {\bigl\{#1\bigr\}}

\newcommand \Jir [1] {\textup{Ji}\,#1} 
\newcommand \Mir [1] {\textup{Mi}\,#1}

\newcommand \then {\mathrel{\Rightarrow}} 
\newcommand\init [1] {} 
\newcommand\nothing [1] {}

\begin{document}
\title[Quasiplanar diagrams and slim semimodular lattices]
{Quasiplanar diagrams and slim semimodular lattices}

\author[G.\ Cz\'edli]{G\'abor Cz\'edli}
\email{czedli@math.u-szeged.hu}
\urladdr{http://www.math.u-szeged.hu/$\sim$czedli/}
\address{University of Szeged, Bolyai Institute. 
Szeged, Aradi v\'ertan\'uk tere 1, HUNGARY 6720}

\thanks{This research was supported by the NFSR of Hungary (OTKA), grant numbers  K77432 and K83219}


\nothing{
05E99 Algebraic combinatorics (1991-now) None of the above, but in this section;
06C10 (1980-now) Semimodular lattices, geometric lattices
06A06 (1991-now) Partial order, general
06A07 (1991-now) Combinatorics of partially ordered sets 
}


\date{December 31, 2012}

\begin{abstract} 
A (Hasse) diagram of a finite partially ordered set (poset) $P$ will be called \emph{quasiplanar} if for any two incomparable elements $u$ and $v$, either $v$ is on the left of all maximal chains containing $u$, or $v$ is on the right of all these chains. Every planar diagram is quasiplanar, and $P$ has a quasiplanar diagram if{f} its order dimension is at most 2. A finite lattice is slim if it is join-generated by the union of two chains. We are interested in diagrams only up to similarity.   The main result gives a bijection between the set of the (similarity classes of) finite quasiplanar diagrams  and that of the (similarity classes of) planar diagrams of finite, slim, semimodular lattices. This bijection allows one to describe finite posets of order dimension at most 2 by finite, slim, semimodular lattices, and conversely. As a corollary, we obtain that there are exactly $(n-2)!$ quasiplanar diagrams of size n.
\end{abstract}

\maketitle
\section{Introduction}\label{introsection}
\subsection{Motivation and aim}
Our original goal was to describe finite, slim, semimodular lattices $L$ by the posets (partially ordered sets) $\Mir L=\tuple{\Mir L;\leq}$ of their meet-irreducible elements. This was motivated by three facts:  there are many results on lattices with unique meet irreducible decompositions,  slim semimodular lattices have intensively been studied recently, and it is well-known that  finite distributive lattices can be described this way.

\init{R.\,P.\ }Dilworth \cite{r:dilworth40} was the first to deal with \emph{unique meet irreducible decompositions} in finite lattices. To give a brief overview, let $\upcover x$ denote the join of all covers of $x$ in a finite lattice $L$. If the interval $[x,\upcover x]$ is distributive for all $x\in L$, then $L$ is a \emph{join-distributive} lattice in nowaday's terminology. There are more than a dozen equivalent definitions of these lattices and two equivalent concepts, antimatroids and convex geometries. \init{R.\,P.\ }Dilworth \cite{r:dilworth40}, who was the first to consider  these lattices,  used  the (equivalent) definition that each element 
can uniquely be decomposed into a meet of meet irreducible elements. The early variants were surveyed in \init{B.\ }Monjardet \cite{monjardet}.  Since it would wander to far if we overviewed the rest, we only mention \init{K.\ }Adaricheva \cite{adaricheva},
\init{H.\ }Abels \cite{abels}, \init{N.\ }Caspard and \init{B.\ }Monjardet \cite{caspardmonjardet}, \init{S.\,P.\ }Avann \cite{avann}, 
\init{R.\,E.\ }Jamison-Waldner  \cite{jamisonwaldner}, and \init{M.\ }Ward \cite{ward}  for additional sources, and \init{M.\ }Stern\ \cite{stern}, 
\init{K.\ }Adaricheva and \init{G.\ }Cz\'edli \cite{adarichevaczg}, and 
\init{G.\ }Cz\'edli \cite{czgcoord} for some recent overviews. However, the reader is not assumed to be familiar with these sources since the present paper is intended to be self-contained for those who know the rudiments of 
Lattice Theory  up to, say, the Jordan-H\"older Theorem for semimodular lattices.
What is mainly important for us is that slim semimodular lattices, to be defined soon, are known to be  join-distributive, see \init{G.\ }Cz\'edli, {L.\ }Ozsv\'art, and  \init{B.\ }Udvari \cite[Corollary 2.2]{czgolub}.

A finite lattice $L$ is \emph{slim}, if $\Jir L$, the set of nonzero join-irreducible elements of $L$, is included in the join of two appropriate chains of $L$; see \init{G.\ }Cz\'edli and \init{E.\,T.\ }Schmidt \cite{czgschtJH}. In the semimodular case, this concept was introduced by \init{G.\ }Gr\"atzer and \init{E.\ }Knapp \cite{gratzerknapp} in a slightly different way. The  theory of slim semimodular lattices has developed a lot recently, as witnessed by 
\init{G.\ }Cz\'edli \cite{czgmatrix}, \cite{czgrepres}, and \cite{czgasympt}, 
\init{G.\ }Cz\'edli, \init{T.\ }D\'ek\'any,  \init{L.\ }Ozsv\'art, \init{N.\ }Szak\'acs, and \init{B.\ }Udvari \cite{czgdekanyatall}, 
\init{G.\ }Cz\'edli and \init{G.\ }Gr\"atzer \cite{czgggresect}, 
\init{G.\ }Cz\'edli, {L.\ }Ozsv\'art, and  \init{B.\ }Udvari \cite{czgolub},
\init{G.\ }Cz\'edli and \init{E.\,T.\ }Schmidt \cite{czgschtJH}, \cite{czgschvisual}, \cite{czgschperm}, and \cite{czgschslim2}, 
\init{G.\ }Gr\"atzer and \init{E.\ }Knapp \cite{gratzerknapp}, \cite{gratzerknapp3}, and \cite{gratzerknapp4}, and \init{E.\,T.\ }Schmidt \cite{schmidtrepres}. In particular, \cite{czgschtJH} gives an application of these lattices outside Lattice Theory while \cite{czgmatrix}, \cite{czgggresect},  \cite{czgschvisual}, \cite{czgschperm}, \cite{czgschslim2}, and \cite{gratzerknapp}, partly of fully, are devoted to their structural descriptions.

All lattices and posets in the paper are assumed to be \emph{finite}, even if this convention is not repeated all the time. We have already mentioned that slim semimodular lattices are join-distributive. This fact, combined with Dilworth's original definition of these lattices, and some recent propositions in \init{G.\ }Cz\'edli \cite{czgcircles} led to our original goal, mentioned at the beginning of the paper.  Since the poset $\Mir L$ does not determine a slim, semimodular lattice $L$ in general, the original target had to be modified.

Slim lattice are planar by \init{G.\ }Cz\'edli and \init{E.\,T.\ }Schmidt \cite[Lemma 2.1]{czgschtJH}, that is, they allow planar (Hasse) diagrams. Although the corresponding posets $\Mir L$ are not planar in general, their appropriate diagrams still have an important property of planar ones; we will coin the name \emph{quasiplanar} to this property. For a first impression, note that all diagrams but $Q_5$ in Figures~\ref{fig2} and \ref{fig3} are quasiplanar; in particular, $Q_3$ is quasiplanar but not planar. 
Now, the modified target is to describe the \emph{planar diagrams} of slim semimodular lattices by quasiplanar diagrams. Of course, diagrams are only considered up to similarity, to be defined  soon. 
The main result of the paper, Theorem~\ref{thmmain}, gives a canonical bijection between the class of planar diagrams of slim semimodular lattices and that of quasiplanar diagrams. This way even the original goal is achieved in a weak sense, because $L$ is described by any of its planar diagram $D$, and $D$ described by a quasiplanar diagram, which is much smaller than $D$ in general. Note that the converse possibility offered by Theorem~\ref{thmmain}, that is the description of quasiplanar diagrams by planar diagrams of slim, semimodular lattices, could also be interesting, because  slim semimodular lattices are well-studied. The strength of this converse option will be demonstrated by Corollary~\ref{corollcount}, which counts quasiplanar diagrams of a given size.

\subsection{Outline}
After recalling or introducing the necessary concepts, 
Section~\ref{conceptsection} formulates the  main result, Theorem~\ref{thmmain}, which asserts that finite, slim, semimodular lattice diagrams and finite quasiplanar diagrams mutually determine each other. Also, this section gives the exact number of $n$-element quasiplanar diagrams, see Corollary~\ref{corollcount}. Section~\ref{proofsection}, which contains many auxiliary statements, is devoted to the proof of Theorem~\ref{thmmain}. Finally, 
Section~\ref{commentsection} contains some comments that shed more light on the main result.

\subsection{Prerequisites} As mentioned already, the reader is not assumed to have deep knowledge of semimodular lattices; a little part of  any book on lattices or particular lattices, including \init{G.\ }Gr\"atzer \cite{GGLT}, \init{J.\,B.\ }Nation \cite{nationbook}, and \init{M.\ }Stern \cite{stern}, is sufficient.

\section{Some concepts and the main result}\label{conceptsection}
\subsection{Quasiplanar diagrams}
A (Hasse) diagram $D$  of a poset $P=\tuple{P;\leq}$ consists of some \emph{points} on the plane, representing the elements of $P$, and \emph{edges}, which are non-horizontal straight line segments connecting two points and represent the covering relation in $P$ in the usual way. Concepts and properties originally defined for posets (and lattices if $P$ happens to be lattice) will also be used for their diagrams; for example, we can speak of a maximal chain of a diagram, and we can say that a lattice diagram is slim and semimodular. A diagram is \emph{planar} if its edges do not intersect, except possibly at their endpoints.  For a more exact definition of planarity and the concepts defined in the next paragraph, the reader can (but need not) resort to  \init{D.\ }Kelly and \init{I.\ }Rival \cite{kellyrival}.

Let $C$ be a maximal chain in a diagram $D$.  This  chain cuts $D$ into
a \emph{left side} and a \emph{right side},  see \init{D.\ }Kelly and \init{I.\ }Rival \cite[Lemma 1.2]{kellyrival}.  (This is so even if $D$ is not planar.) The
intersection of these sides is $C$. If $x\in D$ is on the left side of $C$ but not in $C$, then $x$ is \emph{strictly on the left} of $C$. Let $E$ be another maximal chain of $D$. If all elements of $E$ are on the left of $C$, then $E$ is \emph{on the left of} $C$. In this sense, we can speak of the leftmost maximal chain of $L$, called the \emph{left boundary chain}, and the rightmost maximal chain, called the \emph{right boundary chain}. The union of these two chains is the \emph{boundary} of $L$.
Also, if $F$ is a (not necessarily maximal) chain of $D$, then the \emph{leftmost maximal chain through} $F$ (or extending $F$) and the rightmost one make sense. If $F=\set{f_1<\dots <f_n}$, then the leftmost maximal chain of $D$ through $F$ is the union of the left boundary chains of the subdiagrams  $\ideal {f_1}=\set{x\in D: x\leq f_1}$, $[f_1,f_2]$, \dots, $[f_{n-1},f_n]$, and $\filter{f_n}= \set{x\in D: x\geq f_n}$.  If $F=\set f$ is a singleton, then  chains containing $f$ are said to be chains through $f$ rather than chains through $\set f$. 
The  most frequently used results of \init{D.\ }Kelly and \init{I.\ }Rival  \cite{kellyrival} are the following two.

\begin{lemma}[{\init{D.\ }Kelly and \init{I.\ }Rival  \cite[Lemma 1.2]{kellyrival}}]
\label{krchainlemma}
Let $D$ be a finite, planar lattice diagram, and let $x\leq y\in D$. If $x$ and $y$ are on different sides of a maximal chain $C$ in $L$, then there exists an element $z\in C$ such that $x\leq z\leq y$.
\end{lemma}

\begin{lemma}[{\init{D.\ }Kelly and \init{I.\ }Rival  \cite[Propositions 1.6 and 1.7]{kellyrival}}]\label{leftrightlemma} Let $D$ be finite, planar lattice diagram, and let $x,y\in L$ be incomparable elements. If $x$ is on the left of some maximal chain $($of $D)$  through $y$, then $x$ is on the left of \emph{every} maximal chain through $y$.
\end{lemma}

We will only consider \emph{bounded diagrams}, that is diagrams with 0 and 1, because otherwise the meaning of the left or right side of a maximal chain, which is possibly a singleton, is less pictorial. 
Note, however, that this paper could easily be translated to the ``not necessarily bounded setting'' by defining quasiplanar diagrams as $P\setminus\set{0,1}$ subdiagrams of bounded quasiplanar diagrams $P$. Let us emphasize that a quasiplanar diagram always has 0 and 1 by definition.
By the following definition, 
Lemma~\ref{leftrightlemma} will hold but Lemma~\ref{krchainlemma} may fail for those poset diagrams that play a crucial role in the paper.

\begin{definition}\  
\begin{enumeratei}
\item A diagram $D$ is \emph{quasiplanar} if it is bounded and, in addition,  for any two incomparable $x,y\in D$, whenever $x$ is on the left of \emph{some} maximal chain  through $y$, then $x$ is on the left of \emph{every} maximal chain through $y$.
\item For $x,y$ in a quasiplanar diagram $D$, $x$ is \emph{on the left} of $y$, in notation $x\lft y$, if $x\parallel y$ and $x$ is on the left of some (equivalently, every) maximal chain through $y$. The relation $x\rght y$, worded as $x$ is \emph{on the right} of $y$, is defined analogously.
\end{enumeratei}
\end{definition}
Let us emphasize that whenever left, right, $\slft$, or $\srght$ is used for two elements, then the elements in question are incomparable. Therefore, for example, the implication $x\lft y\then x\parallel y$ holds throughout the paper. 
By Lemma~\ref{leftrightlemma}, every planar lattice diagram is quasiplanar. Since planar bounded diagrams are lattice diagrams by \init{D.\ }Kelly and \init{I.\ }Rival  \cite[Corollary 2.4]{kellyrival}, a planar bounded diagram is necessarily quasiplanar.
The following statement is an obvious extension of Proposition 1.7 in 
\init{D.\ }Kelly and \init{I.\ }Rival  \cite{kellyrival}; its last part follows by considering a maximal chain through $\set{y,a}$.

\begin{lemma}\label{lfrhtrRlemma}
Let $x,y,z,a$, $b$, and $c$ be elements of a quasiplanar diagram. Then the following hold.
\begin{enumeratei}
\item\label{lfrhtrRlemmaa} If $x\lft y$ and $y\lft z$, then $x\lft z$.
\item\label{lfrhtrRlemmab} If $a\parallel b$, then either $a\lft b$, or $b\lft a$. 
\item\label{lfrhtrRlemmac} If $x\lft y$ and $a\nonparallel y$, then either $x\lft a$, or $a\nonparallel x$.
\end{enumeratei}
\end{lemma}

If $D_1$ and $D_2$ are quasiplanar diagrams and there exists a bijection $\psi\colon D_1\to D_2$ such that $\psi$ is an order isomorphism and, for any $x,y\in D_1$, $x\lft y$ in $D_1$ if{f} $\psi(x)\lft\psi(y)$ in $D_2$, then $D_1$ and $D_2$ are \emph{similar diagrams}
and $\psi$ is a \emph{similarity map}.  For lattice diagrams, 
similarity means the same as in \init{D.\ }Kelly and \init{I.\ }Rival  \cite{kellyrival}. We consider quasiplanar diagrams up to similarity; that is, similar diagrams will always be treated as equal ones, even if this is not repeated all the time. An important tool to recognize similarity is given in the following lemma, which is taken from \init{G.\ }Cz\'edli and \init{E.\,T.\ }Schmidt  and \cite[Lemma 4.7]{czgschslim2} or, more explicitly, 
\init{G.\ }Cz\'edli and \init{G.\ }Gr\"atzer \cite{czgggltsta}.

\begin{lemma}\label{lemmaleftbounddeterm} Let $D_1$ and $D_2$ be slim,  semimodular lattice diagrams. If there exists an order-isomorphism $\psi\colon D_1\to D_2$ such that $\psi$ maps the left boundary chain of $D_1$ to the left boundary chain of $D_2$, then 
$D_1$ and $D_2$ are similar diagrams and $\psi$ is a \emph{similarity map}. 
\end{lemma}

\subsection{The key constructions}
Before formulating the main result, we have to give the basic constructions. It is not so trivial that our constructs exist and have the desired properties, but this will be proved later, in due time.

\begin{definition}\label{alfaDdef} Let $D$ and $Q$  be a planar lattice diagram and a quasiplanar diagram, respectively. We say that $Q$ is the \emph{quasiplanar diagram associated} with $D$ if the following hold.
\begin{enumeratei}
\item $Q=\set{1,\nnul}\cup \Mir D$, where $1\in D$, $\nnul\notin D$;
\item for $x,y\in Q$,  $x\leq y$ in $Q$ if{f} $x\leq y$ in $D$ or $x=\nnul$;
\item for any two incomparable $x,y\in Q$, we have $x\lft y$ in $Q$ if{f} $x\lft y$ in $D$.
\end{enumeratei}
If $Q$ above exists, then it is clearly unique up to similarity; it is denoted by $\alpha(D)$.
\end{definition}

We do not claim that $Q$ above exists for every $D$. As usual, the equality relation on a diagram $Q$ is denoted by $\omega_Q$.  If $\slft=\slft_Q$ is the relation ``on the left'' on $Q$, then $\eslft$ denotes the relation $\slft\cup\omega_Q$, and we also have $\esrght=\srght\cup\omega_Q$. In particular, $x\elft y$ means that either $x=y$, or $x\parallel y$ and $x$ is on the left of $y$. We define the relations $\leqslft$, $\geqslft$, $\leslft$,  $\bigglft$, 
$\leqrght$, $\geqrght$, $\lesrght$, and $\biggrght$ analogously; for example, 
$x\leqlft y$ means that $x\leq y$ or $x\lft y$, and   $x\biggrght y$ means $x> y$ or $x\rght y$.
Next, we start from a quasiplanar diagram, and want to define a slim semimodular lattice diagram. 

\begin{definition}\label{defbetaegy} For a quasiplanar diagram $Q$, let $\plu Q=Q\setminus\set{0}$. A pair $\pair xy\in \plu Q\times \plu Q$ is a \emph{$\eslft$-pair} if 
$x\elft y$. The set of these pairs is denoted by $\Elig(Q)$. (At set theoretical level, $\Elig(Q)$ is the same as $\eslft$.)
For $\pair{x_1}{y_1}, \pair{x_2}{y_2}\in \Elig(Q)$, we define  
{\allowdisplaybreaks{
\begin{align}
\label{defbetaegya} 
\pair{x_1}{y_1} \leq \pair{x_2}{y_2} &
\overset{\text{def}}\Longleftrightarrow 
x_1 \leqlft x_2\text{ and }y_2\geqlft y_1\text{, and}\\
\label{defbetaegyb} 
\pair{x_1}{y_1} \lft \pair{x_2}{y_2} &
\overset{\text{def}}\Longleftrightarrow  x_1 \leslft x_2\text{ and }y_1\bigglft y_2\text.
\end{align}}}%
(Note that $y_2\geqlft y_1$ in \eqref{defbetaegya}  is equivalent to $y_1\leqrght y_2$.)
Let $\beta_1(Q)$ be the unique  planar diagram of $\tuple{\Elig(Q);\leq}$, where ``$\leq$'' is given by \eqref{defbetaegya}, such that the ``on the left'' relation of $\beta_1(Q)$ is described by \eqref{defbetaegyb}. (We will prove that such a diagram exists; its uniqueness is obvious.)
\end{definition}

The advantage of Definition~\ref{defbetaegy} is that $\eslft$-pairs are relatively simple objects and $\lft$ in $\beta_1(Q)$ is quite explicitly described. However,
we will also benefit from the 
 the following approach in our proofs.

\begin{definition}\label{defbetaket} Let $Q$ be a quasiplanar diagram, and let $\plu Q=Q\setminus\set 0$.
\begin{enumeratei}
\item\label{defbetaketa} A nonempty subset $X$ of $\plu Q$ is called a \emph{proper horizontally convex order filter}, in short a \emph{hco-filter}, of $Q$ if 
\begin{itemize}
\item  $X$ is an up-set, that is, $x\in X$, $y\in Q$, and $x\leq y$ implies $y\in X$, and 
\item $X$ is horizontally convex, that is, if $x\lft y$, $y\lft z$, and $\set{x,z}\subseteq X$, then $y\in X$.
\end{itemize}
\item\label{defbetaketb}  For $Y\subseteq \plu Q$, the least hco-filter including $Y$ is denoted by $\hcofilter Y=\phcofilter QY$; we write $\hcofilter y$ instead of $\hcofilter{\set y}$. 
\item\label{defbetaketc}  The set of hco-filters of $Q$ is denoted by $\Hco Q$. For $X,Y\in\Hco Q$, let $X\dleq Y$ mean $X\supseteq Y$; the poset $\tuple{\Hco Q;\dleq}$ is also denoted by $\Hco Q$.
\item\label{defbetaketd}  We define a finite sequence of hco-filters $\vec F(Q)=\vec F=(F_0,F_1, \ldots, F_{|Q|-2})$ by induction as follows. Let $F_0=\set 1$. If $F_n$ is defined and $\plu Q\setminus F_n\neq\emptyset$, then
let $f_n$ be the leftmost element
in the set $\Max(Q\setminus F_n)$ of maximal elements of $Q\setminus F_n$, and let $F_{n+1}=F_n\cup\set {f_n}$.
\item\label{defbetakete}  We also define the ``left-right dual'' version $\vec G(Q)=\vec G=(G_0,G_1, \ldots, G_{|Q|-2})$ of $\vec F$ by induction as follows.   Let $G_0=\set 1$. While $\plu Q\setminus G_n\neq\emptyset$, denote 
by $g_n$  the rightmost element
in the set $\Max(Q\setminus G_n)$ of maximal elements of $Q\setminus G_n$, and let $G_{n+1}=G_n\cup\set {G_n}$.
\item\label{defbetaketf}  Let $\beta_2(Q)$ be the planar lattice diagram of $\tuple{\Hco Q;\dleq}$ such that $\vec F$ and $\vec G$ are the left boundary chain and the right boundary chain, respectively. (We will show later that this makes sense.)
\end{enumeratei}
\end{definition}

\subsection{The results}
In order to take  Definitions~\ref{defbetaegy} and \ref{defbetaket} into account independently, the main theorem below contains a parameter $p\in\set{1,2}$.

\begin{theorem}[Main Theorem]\label{thmmain}
Let $D$ be a finite, slim, semimodular lattice diagram, and let $Q$ be a finite quasiplanar diagram. Let $p\in\set{1,2}$. Then the following hold.
\begin{enumeratei}
\item\label{thmmaina} $\alpha(D)$ is a finite quasiplanar diagram.
\item\label{thmmainb} $\beta_p(Q)$ is a finite, slim, semimodular lattice diagram.
\item\label{thmmainc} Up to similarity, $\beta_p(\alpha(D))$ equals $D$.
\item\label{thmmaind} Up to similarity, $\alpha(\beta_p(Q))$ equals $Q$.
\end{enumeratei}
\end{theorem}

We now from \init{G.\ }Cz\'edli and \init{E.\,T.\ }Schmidt~\cite{czgschperm}, see also \init{G.\ }Cz\'edli and \init{G.\ }Gr\"atzer \cite{czgggltsta}, that there exists a bijection between the set of slim semimodular lattice diagrams of length $n$ and the set $S_n$ of permutations acting on $\set{1,\ldots,n}$. Therefore, Theorem~\ref{thmmain} immediately implies the following statement. Let us emphasize that quasiplanar diagrams are bounded by definition.

\begin{corollary}\label{corollcount} Up to similarity, the number of $n$-element quasiplanar diagrams is $(n-2)!$.
\end{corollary}

\section{Auxiliary statements and proofs}\label{proofsection}

\subsection{Statements on quasiplanar diagrams}
Let $Q$ be a quasiplanar diagram, and  let $F\in \Hco Q$ be a hco-filter.
The set of minimal elements of $F$ is denoted by $\Min F$. It is an antichain, so it has a unique leftmost element $\lmost F$, and a unique rightmost element $\rmost F$. The are called the \emph{leftmost bottom element} and the \emph{rightmost bottom element} of $F$, respectively. Clearly, $\lmost F\elft \rmost F$. If $\pair xy\in\Elig(Q)$, then we often use the following notation
\begin{align*}
\Betw xy&=\set{z: x\elft z\text{ and }z\elft y} \text{ and }\cr
\MinBetw xy&=\Min\set{z: x\elft z\text{ and }z\elft y},
\end{align*}
where for an $A\subseteq Q$, $\Min A$ denotes the set of minimal elements of $A$. Since $x\elft y$, the set $\MinBetw xy$ is not empty. 
For $U\subseteq Q$, $\filter U$ denotes the order filter $\set{z\in Q: z\geq u\text{ holds for some }u\in U}$ generated by $U$.

\begin{lemma}\label{sdGT} If $Q$ is a quasiplanar diagram, then for any $\pair{x}{y}\in \Elig(Q)$, we have 
\begin{align}
&\hcofilter{\set{x,y}}=\filter \MinBetw xy\text{ and, in particular, }\hcofilter x=\filter x;  \label{sdGTa}\\
&x=\lmost{\hcofilter{\set{x,y}}}, \quad  \quad y=\rmost{\hcofilter{\set{x,y}}};\label{sdGTb}\\
&\Min (\hcofilter{\set{x,y}}) = \MinBetw xy\text.\label{sdGTc}
\end{align} 
\end{lemma}

\begin{proof} The ``$\supseteq$''  inclusion in the first equation of \eqref{sdGTa} is obvious. Assume that $u_1,u_2\in \filter \MinBetw xy$, $u\in Q$, and $u_1\lft u\lft u_2$.
We want to show $u\in \filter \MinBetw xy$.   
There are $v_1,v_2\in \MinBetw xy$ such that $v_1\leq u_1$ and $v_2\leq u_2$. By Lemma~\ref{lfrhtrRlemma}\eqref{lfrhtrRlemmac}, either $u\lft v_2$ or $u\nonparallel v_2$. Now $u\leq v_2$ would give $u\leq u_2$, which would contradict $u\lft u_2$. If we had $u\geq v_2$, then $u\in  \filter \MinBetw xy$ would trivially hold. Hence we can assume $u\lft v_2$.
Similarly, we can also assume $v_1\lft u$. We know that $x\lft v_1$ and $v_2\lft y$. Armed with the formulas $x\lft v_1$, $v_1\lft u$, $u\lft v_2$, and $v_2\lft y$, Lemma~\ref{lfrhtrRlemma} yields 
$u\in \Betw xy\subseteq  \filter \MinBetw xy$. Therefore, $\filter \MinBetw xy$ is a hco-filter. Finally, it is trivial that $x$ and $y$ belong to 
$\set{z: x\elft z\text{ and }z\elft y}$, and they are minimal elements in this set. That is, $\set{x,y}\subseteq \MinBetw xy$, and the ``$\subseteq$'' inclusion in \eqref{sdGTa} follows. This proves the first equation of \eqref{sdGTa};  the second one is a particular case since $\pair xx\in\Elig(Q)$.

Obviously, if $A$ is an antichain, then $\Min {(\filter A)}=A$. Applying this fact to $A=\MinBetw xy$ and taking \eqref{sdGTa} into account, we conclude \eqref{sdGTb} and \eqref{sdGTc}.
\end{proof}

The following lemma says that Definitions~\ref{defbetaegy} and \ref{defbetaket} are quite close to each other. 

\begin{lemma}\label{phipilemMa} Given a quasiplanar diagram $Q$, the maps
\[\eptofilt\colon \Elig (Q)\to \Hco Q,\quad\text{defined by }\pair xy\mapsto \hcofilter{\set{x,y}},\]
and 
\[\filttoep\colon \Hco Q\to \Elig (Q),\quad\text{defined by }F\mapsto \pair{\lmost F}{\rmost F},\]
are reciprocal order isomorphisms.
\end{lemma}

\begin{proof} 
Assume that 
$\pair{x_1}{y_1}\leq \pair{x_2}{y_2}$ in 
 $\Elig(Q)$. This means that $x_1\leqlft x_2$
and $y_2 \geqlft y_1$. Let $F_i= \hcofilter{\set{x_i,y_i}} = \eptofilt(\pair{x_i}{y_i})$ for $i\in\set {1,2}$.  
To obtain  $F_1 \dleq F_2$, that is $F_2\subseteq F_1$, we have to show $x_2,y_2\in \hcofilter{\set{x_1,y_1}}$. We can assume $x_1\not\leq x_2$ since otherwise $x_2 \in \hcofilter{\set{x_1,y_1}}$ trivially holds. Thus $x_1\lft x_2$. If $y_2\lft y_1$, then $x_1\lft x_2\elft y_2\lft y_1$, together with the horizontal convexity of $F_1$, yields $x_2\in F_1$. If $y_2\geq y_1$, then $y_2\in F_1$, $x_1\lft x_2\elft y_2$, and the horizontal convexity of $F_1$ yield $x_2\in F_1$ again. Hence, $x_2\in F_1$, and $y_2\in F_1$ follows  by left-right duality. Therefore, $\eptofilt$ is order-preserving. 

We know from Lemma~\ref{sdGT}\eqref{sdGTb} that $\filttoep\circ\eptofilt$ is the identity $\Elig(Q)\to\Elig(Q)$ map. To prove that $\eptofilt\circ\filttoep$ is the identity $\Hco Q\to \Hco Q$ map, let $F\in \Hco Q$. Denoting $\lmost F$ and $\rmost F$ by $x$ and $y$, respectively, we have 
$\filttoep(F)=\pair xy$. We also have
$(\eptofilt\circ \filttoep)(F)= \eptofilt(\filttoep(F))= \hcofilter\set{x,y}$. The inclusion $(\eptofilt\circ \filttoep)(F)= \hcofilter\set{x,y}\subseteq F$ is trivial. To show the converse inclusion, let $u\in F$. Then there exists a $v$ in the antichain $ \Min F$ such that $u\geq v$. By the definition of $x$ and $y$, we have $x\lft v\lft y$. Hence $v\in \hcofilter\set{x,y}$, which implies $u\in \hcofilter\set{x,y}$. This proves that $\eptofilt\circ\filttoep$ is the identity $\Hco Q\to \Hco Q$ map, and thus $\eptofilt$ and $\filttoep$ are reciprocal bijections. 

Finally, to prove that $\filttoep$ is order-preserving, assume that $F_1\dleq F_2\in \Hco Q$. Denoting $\filttoep(F_i)$ by $\pair{x_i}{y_i}$, this means $\hcofilter\set{x_1,y_1}\supseteq \hcofilter\set{x_2,y_2}$. Hence, by \eqref{sdGTa}, $\set{x_2,y_2}\subseteq \filter{\MinBetw {x_1}{y_1}}$. 
If $x_2\geq x_1$, then $x_1\leqlft x_2$ is clear. Hence, we assume $x_2\not\geq x_1$.  It follows trivially or from Lemma~\ref{sdGT}\eqref{sdGTc} that $x_1$ belongs to the set $\Min( \filter{\MinBetw {x_1}{y_1}} )$, whence  $x_2\not<x_1$. Thus $x_2\parallel x_1$. By Lemma~\ref{sdGT}\eqref{sdGTc}, there exists a $u\in\MinBetw xy$ such that $x_2\geq u$, and we obtain $x_1\lft x_2$  from Lemma~\ref{lfrhtrRlemma}\eqref{lfrhtrRlemmac}. Hence, in all cases, $x_1\leqlft x_2$. By left-right duality, we obtain $y_2\geqlft y_1$. Therefore, $\filttoep(F_1)=\pair{x_1}{y_1}\leq \pair{x_2}{y_2}=\filttoep(F_2)$.
\end{proof}

The concept of antimatroids is due to \init{R.\,E.~}Jamison-Waldner~\cite{jamison}. Like in \init{G.\ }Cz\'edli \cite{czgcoord}, we again cite the following definition from  \init{D.~}Armstrong~\cite[Lemma 2.1]{armstrong}.
The set of all subsets of a set $E$ is denoted by $\powset E$.

\begin{definition}
A pair $\pair{E}{\alg F}$ is an \emph{antimatroid} if it satisfies the following properties:
\begin{enumeratei}
\item\label{antimatdefa} $E$ is a finite set, and $\emptyset\neq \alg F \subseteq \powset E$;
\item\label{antimatdefb} $\alg F$ is a \emph{feasible set}, that is, for each nonempty $A\in \alg F$, there exists an $x\in A$ such that $A\setminus\set x\in\alg F$;
\item\label{antimatdefc} $\alg F$ is closed under taking unions;
\item\label{antimatdefd} $E=\bigcup\set{A: A\in \alg F}$.  
\end{enumeratei}
\end{definition}

The relevance of this concept here
is explained by the following well-known statement; see 
\init{D.\ }Armstrong~\cite[Theorem 2.6]{armstrong}, who attributes it to 
\init{G.\ }Birkhoff, Whitney and \init{S.\ }MacLane, or \init{K.\ }Adaricheva, \init{V.\,A.\ }Gorbunov,  and \init{V.\,I.\ }Tumanov~\cite{adarichevaetal}, 
  see also \init{G.\ }Cz\'edli~\cite{czgcoord}.
    
\begin{lemma}\label{antimtrJDl} If $\pair{E}{\alg F}$ is an antimatroid, then $\tuple{\alg F;\subseteq}$ is a finite join-dist\-rib\-utive lattice. Up to isomorphism, each join-distributive lattice can be obtained this way. 
\end{lemma}

\begin{lemma}\label{nhVbztZTsV}
If $Q$ is a quasiplanar diagram, then $\tuple{\Hco Q;\dleq}$ is a semimodular lattice.
\end{lemma}

\begin{proof} Let $\alg F=\set{\plu Q\setminus F: F\in \Hco Q }$ and $E=Q\setminus\set{0,1}$.  
Then $\alg F\subseteq \powset E$ and 
$\tuple{\Hco Q;\dleq}\cong \tuple{\alg F ;\subseteq}$. Since $\Hco Q$ is clearly closed with respect to intersections, $\alg F$ is closed under taking unions. We claim that $\tuple{E;\alg F }$ is an antimatroid.  This will prove Lemma~\ref{nhVbztZTsV}, because then  Lemma~\ref{antimtrJDl} applies and join-distributive lattices are semimodular; 
see, for example, \init{B.\ }Monjardet \cite{monjardet}, \init{R.\,E.\ }Jamison-Waldner  \cite{jamisonwaldner},   and see \cite{armstrong}, \cite{adarichevaetal}, and \cite{czgcoord} mentioned a few lines above.  
Since $E$ and $\emptyset$ belong to $\alg F$, we only have to show that $\alg F$ is a feasible set. By the definition of $\alg F$, is suffices to prove that 
%
if $\plu Q\neq F\in \Hco Q$, then there exists an element $u$ in $\plu Q\setminus F$ such that $F\cup\set u\in\Hco Q$.
To show this, take a minimal $G\in  \Hco Q$, with respect to ``$\subseteq$'', such that $F\subset G$; it is sufficient to prove that $|G\setminus F|=1$.
By Lemma~\ref{phipilemMa}, $F$ is of the form $\hcofilter{\set{x,y}}$ for some $\pair xy\in\Elig(Q)$. 
There are three cases to discuss, but first we formulate the following three rules. 
{\allowdisplaybreaks{
\begin{align}
\label{twoenough}
(\forall x_1,x_2,x_3\in \plu Q)\,\, (\exists i\in\set{1,2,3})\,\,(x_i\in \hcofilter{(\set{x_1,x_2,x_3}\setminus\set{x_i})}  ;\\
\label{sztnBnRw}
(\pair {x_1}{x_2}\in \Elig(Q) \text{ and } x_3<x_1 ) \,\,\then\,\, x_3\notin\hcofilter\set{x_1,x_2};\\
\label{szlsnmjtszk}
x_1\lft x_2\lft x_3\,\,\then\,\,(x_1\notin \hcofilter\set{x_2,x_3} \text{ and }  x_3\notin \hcofilter\set{x_1,x_2} )\text.
\end{align}}}%
The validity of \eqref{twoenough} is obvious if $\set{x_1,x_2,x_3}$ is not a three-element antichain, and it follows from the fact that one of the three elements is horizontally between the other two otherwise. 
To prove \eqref{sztnBnRw} by way of contradiction, suppose that \eqref{sztnBnRw} fails. Then $x_1\lft x_2$, and Lemma~\ref{sdGT}\eqref{sdGTa} yields a $t$ such that $x_1\elft t\elft x_2$ and $t\leq x_3$. We have $x_1\neq t$ since $x_3<x_1$. Hence $x_1> x_3\geq t$ contradicts $x_1\lft t$, proving \eqref{sztnBnRw}. Next, it suffices only to prove \eqref{szlsnmjtszk} for $x_1$, because then the $x_3$-part follows by left-right symmetry. By way of contradiction, suppose $x_1\lft x_2\lft x_3$ but $x_1\in \hcofilter\set{x_2,x_3}$. By Lemma~\ref{sdGT}\eqref{sdGTa}, there exists a $t$ such that $x_2\elft t\elft x_3$ and $t\leq x_1$. Actually, $x_2\lft t\lft x_3$ since $\set{x_1,x_2,x_3}$ is an antichain. We obtain $x_2\lft x_1$ from Lemma~\ref{lfrhtrRlemma}\eqref{lfrhtrRlemmac}, which contradicts $x_1\lft x_2$. This proves \eqref{szlsnmjtszk}.

\begin{caseone} Here we assume that there exists an element 
\[u\in \ideal{\MinBetw xy}\setminus \MinBetw xy = \ideal{\MinBetw xy}\setminus F
\]
such that $u\not\leq z$ for some  $z\in \MinBetw xy$. In what follows, $u$ will stand for such an element. We claim that  $u < x$ or $u<  y$. Suppose the contrary. Then $x\parallel u$, $u\parallel y$, and  there is a $t\in \MinBetw xy$ such that $u<t$. Since $x\lft t\lft y$, Lemma~\ref{lfrhtrRlemma}\eqref{lfrhtrRlemmac} gives $u\in\Betw xy\subseteq F$, a contradiction. Hence, we can assume $u < x$. We claim $u\not\leq y$, and we prove this by way of contradiction. Suppose $u\leq y$. Since $u\parallel z$, either $u\lft z\lft y$ and Lemma~\ref 
{lfrhtrRlemma}\eqref{lfrhtrRlemmaa} yield $u\lft y$, which contradicts $u<y$,   or $x\lft z\lft u$ and we have $x\lft u$, which contradicts $u<x$. Thus $u\not\leq y$. We know $u\not\geq y$ from $u\notin F$. If we had $y\lft u$, then we would obtain $x\lft u$ by Lemma~\ref{lfrhtrRlemma}\eqref{lfrhtrRlemmaa}, which would contradict $u<x$. Therefore, $u\lft y$, and $\pair uy\in\Elig(Q)$.  Clearly, $F\subset \hcofilter{\set{u,y}}\subseteq G$, the minimality of $G$, and Lemma~\ref{sdGT}\eqref{sdGTa} give $G=\hcofilter\set{u,y}=\filter\MinBetw uy$.

We claim  $\MinBetw uy\subseteq \set u\cup\filter\MinBetw xy$. Suppose the contrary. Then there exists a $t\in \MinBetw uy$ such that $u\neq t\notin \filter\MinBetw xy=F$. We have $x>t$, because $t\notin F$ excludes $x\leqlft t$ while $t\lft x$ would lead to 
$t\lft u$ by Lemma~\ref 
{lfrhtrRlemma}\eqref{lfrhtrRlemmac}, which would contradict $u\lft t$. 
Therefore, $G\supseteq \hcofilter\set{t,y}\supset \hcofilter\set{x,y}=F$ and the minimality of $G$ implies that $G=\hcofilter \set{t,y}$. Using $u\in G$ and Lemma~\ref{sdGTa}\eqref{sdGTa}, we obtain an $s\in\Betw ty$ such that $s\geq u$. Hence, Lemma~\ref 
{lfrhtrRlemma}\eqref{lfrhtrRlemmac} yields $t\lft u$ or $t\nonparallel u$, which contradicts $u\lft t$. Consequently, $\MinBetw uy\subseteq \set u\cup\filter\MinBetw xy$.

Next, we claim that, for any $r\in \plu Q$,
\begin{equation}\label{siGlG} r>u \then r\in F\text.
\end{equation}
Suppose the contrary. That is, we have an $r\in G\setminus F$ such that $r>u$. The minimality of $G$ yields $u\in G=\hcofilter\set{r,x,y}$. Since $u\notin F=\hcofilter\set{x,y}$,  \eqref{twoenough} implies $u\in \hcofilter\set{r,x}$ or $u\in \hcofilter\set{r,y}$.
If $r\parallel x$, then \eqref{sztnBnRw} excludes $u\in \hcofilter\set{r,x}$. If $r\nonparallel x$, then $u\in \hcofilter\set{r,x}=\filter r\cup\filter x$ by Lemma~\ref{sdGT}\eqref{sdGTa}, which is excluded by $u<r$ and $u<x$. Hence, $u\in \hcofilter\set{r,y}$. Since $u\parallel y$ and $u<r$, this is excluded if $r\nonparallel y$. Thus $r\parallel y$, and $u\in \hcofilter\set{r,y}$ contradicts \eqref{sztnBnRw}. This proves \eqref{siGlG}.

Finally, combining $\MinBetw uy\subseteq \set u\cup\filter\MinBetw xy$, \eqref{siGlG}, and  $G=\filter\MinBetw uy$, we obtain $G=F\cup\set u$, which gives $|G\setminus F|=1$. 
\end{caseone}

\begin{casetwo}  Here we assume that there exists an element $u\in G\setminus F$ such that  $u\leq z$ for all $z\in {\MinBetw xy}$. (In particular, $u\in \ideal{\MinBetw xy}$.) In what follows, $u$ will stand for such an element.
The minimality of $G$ and Lemma~\ref{sdGT}\eqref{sdGTa} give $G=\hcofilter u=\filter u$. We claim 
\begin{equation}\label{allCovrs}
\text{$u\prec z$ for all $z\in\MinBetw xy$.} 
\end{equation}
To show this by way of contradiction, suppose the contrary. Then there is a $v$ such that $u<v<z$. Since $z$ is a minimal element of $\Betw xy$, Lemma~\ref{sdGT}\eqref{sdGTa} easily implies $v\notin F$. The minimality of $G$ gives 
$u\in G=\hcofilter\set{x,y,v}$. 
We apply \eqref{twoenough} to $\hcofilter\set{x,y,v}$. Since $v\notin F=\hcofilter\set{x,y} =\hcofilter\MinBetw xy  $, left-right symmetry allows us to assume 
$y\in \hcofilter\set{x,v}$. This gives $u\in G=\hcofilter\set{x,v}$. 
Now if we had $x\nonparallel v$, 
then $u\in \hcofilter\set{x,v}=\hcofilter\set{x}\cup  \hcofilter\set{v}=\filter x\cup\filter v$ would contradict $x>u$ and $v>u$. 
Otherwise $\pair xu$ or $\pair vx$ belongs to $\Elig(Q)$, and $u\in\hcofilter\set{x,v}$ contradicts \eqref{sztnBnRw} or the left-right dual of \eqref{sztnBnRw}. This proves \eqref{allCovrs}.
 
Next, we claim
\begin{equation}\label{dTGWRm}
\text{$(\forall z\in\filter u)\,\,(z>u\then z\in F)$.}
\end{equation}
Suppose the contrary, and pick a $v\in\filter u$ such that $v> u$ and $v\notin F=\hcofilter\set{x,y}$.
The minimality of $G$ yields 
$u\in G=\hcofilter\set{x,y,v}$.
By \eqref{twoenough},  $v\notin F$, and left-right symmetry, we can assume $u\in\hcofilter\set{x,v}$. Since $u\prec x$ and $u<v$ exclude $x\nonparallel v$, \eqref{sztnBnRw} yields the same contradiction as in the previous paragraph. 

Finally, \eqref{dTGWRm} and $G=\filter u$ implies $|G\setminus F|=1$.
\end{casetwo}

\begin{casethree}
Here we assume that for all $u\in G\setminus F$,   $u\notin \ideal{\MinBetw xy}$. In what follows, $u$ will stand for such an element of  $G\setminus F$.
Since $u\notin F=\filter{\MinBetw xy}$, the primary assumption of the present case yields that $\set u\cup \MinBetw xy$ is an antichain and $u\notin \Betw xy$. Hence either $u\lft x$ or $y\lft u$; we can assume the latter by left-right symmetry. Since $F=\hcofilter\set{x,y}$ is a proper subset of $\hcofilter\set{x,u}$ by  \eqref{szlsnmjtszk} and $\hcofilter\set{x,u}\subseteq G$, the minimality of $G$ implies $G=\hcofilter\set{x,u}$.  We claim that $u$ is immediately on the right of $y$, that is, 
\begin{equation}\label{ndpwzT}
\text{there is no $v$ such that $y\lft v\lft u$.}
\end{equation}
To prove this by contradiction, suppose the contrary, and take such an element $v$. Since $x\lft v\lft u$ by Lemma~\ref{lfrhtrRlemma}\eqref{lfrhtrRlemmaa}, we have $v\in G$. Also, $F=\hcofilter\set{x,y}\subseteq \hcofilter\set{x,v}$.  But $v\notin F$ and $u\notin \hcofilter\set{x,v}$   by \eqref{szlsnmjtszk}. Hence, $F\subset \hcofilter\set{x,v} \subset G$ contradicts the minimality of $G$. This proves \eqref{ndpwzT}. Next, we claim 
\begin{equation}\label{nnZRQpj}
(\forall v\in Q)\,\,( u<v \then v\in F )\text.
\end{equation}
Suppose the contrary. Then $F=\hcofilter\set{x,y}\subset \hcofilter\set{x,y,v} \subseteq G$, and the minimality of $G$ yields $G=\hcofilter\set{x,y,v}$. Since $G\neq F=\hcofilter\set{x,y}$, \eqref{twoenough}
implies $u\in G= \hcofilter\set{x,v}$ or 
$u\in \hcofilter\set{y,v}$.  If $x\nonparallel v$, then $\hcofilter\set{x,v}=\filter x\cup \filter v$, and $u\parallel x$ and $u<v$ excludes $u\in  \hcofilter\set{x,v}$. If $x\parallel v$, then $u\notin  \hcofilter\set{x,v}$ by 
\eqref{sztnBnRw}. Hence, $u\in \hcofilter\set{y,v}$. We can exclude $y\nonparallel v$ the same way as we excluded $x\nonparallel v$ above. Hence $y\parallel v$, and the left-right dual of \eqref{sztnBnRw} gives a contradiction. This proves \eqref{nnZRQpj}. 

Now we are in the position to show $G=\hcofilter \set{x,u}$ equals $F\cup\set u$. The ``$\supseteq$'' inclusion is clear. To prove the converse inclusion, assume $t\in G\setminus \set u$. By Lemma~\ref{sdGT}\eqref{sdGTa}, there exist a $v\in\MinBetw xu$ such that $v\leq t$. If $v=u$, then $t\in F$ by \eqref{nnZRQpj}. If $v\in F$, in particular, if $v=x$, then $t$ trivially belongs to $F$. Hence, for the sake of contradiction,  suppose $v\notin F$ and $x\lft v\lft u$. 
We claim that there exists a $z \in \MinBetw xy$ such that $v<z$. Suppose the contrary, that is, $v\not< z$ for all $z \in \MinBetw xy$.  Since $v\notin F$, we also have $v\not\geq z$ for all $z \in \MinBetw xy$. Hence $\set {v}\cup\MinBetw xy$ is an antichain. Since \eqref{ndpwzT} and $v\neq u$ exclude $y\lft v$, we have $v\in\Betw xy$. By finiteness, there is a $v'\in\MinBetw xy$ such that $v'\leq v$. But this contradicts $v\notin F$. Therefore, there exists a $z \in \MinBetw xy$ such that $v<z$. Thus we have $v\in G\setminus F$ and $v\in \ideal \MinBetw xy$. This is a contradiction, because we are dealing with Case 3. This proves $G=F\cup \set u$ and $|G\setminus F|=1$.\qedhere
\end{casethree}
\end{proof}

An order filter $F$ of a quasiplanar diagram $Q$ is \emph{left-closed} if for all $x\in F$ and $y\in Q$, $y\lft x$ implies $y\in F$. \emph{Right-closed} order filters $G$ are defined analogously by the property $(x\lft y \text{ and } x\in G) \then y\in G$. 
Clearly, left-closed and right-closed order filters are hco-filters. 
Definition~\ref{defbetaket}\eqref{defbetaketd}-\eqref{defbetakete} should be kept in mind.

\begin{lemma}\label{lMvecFvecBjoIn} If $Q$ is a quasiplanar diagram, then the definition of $\vec F=\vec(Q)$ and that of $\vec G=\vec G(Q)$ make sense. The members of $\vec F$ are left-closed order filters, those of $\vec G$ are right-closed ones, and 
each element of the lattice $\tuple{\Hco Q;\dleq}$ is of the form $F_i\vee G_j$.
\end{lemma}
 
\begin{proof} We prove by induction on $i$ that $F_i$ makes sense and it is a left-closed order filter. This is obvious for $F_0=\set 1$. Assume that $F_n$ is well-defined, it is a left-closed order filter, $|F_n|=n+1$, and $n+2\leq |Q|-2$. Then $\plu Q\setminus F_n\neq\emptyset$. Hence  $\Max(Q\setminus F_n)$ is a antichain, which has a unique leftmost element $f_n$.  We let $F_{n+1}=F\cup\set{f_n}$. It is an order filter, because $f_n$ is a maximal element outside $F_n$. Striving for a contradiction, suppose that $F_{n+1}$ is not left-closed. Then there is an $x\in \plu Q \setminus F_n$ such that $x\lft f_n$. By finiteness, there exists a $u\in \filter x \cap  \Max(Q\setminus F_n)$. Since $x\parallel f_n$, we have $f_n\neq u$, which gives $f_n\lft u$ by the definition of $f_n$. It follows from Lemma~\ref{lfrhtrRlemma}\eqref{lfrhtrRlemmac} that $f_n\lft x$, which contradicts $x\lft f_n$. Consequently, $F_{n+1}$ is a left-closed order filter. This proves that $\vec F$ consists of well-defined left-closed order filters, and left-right duality yields that $\vec G$ consists of right-closed ones. 

Next, let $B\in \Hco Q$. By Lemma~\ref{phipilemMa}, $B=\hcofilter\set{x,y}$ for a unique $\pair xy\in\Elig(Q)$. 
Let $i$ be the least subscript such that 
$y\in F_i$. Similarly, let $j$ be the smallest subscript such that $x\in G_j$. 
We claim $B=F_i\cap G_j$; in the lattice 
$\tuple{\Hco Q;\dleq}$ this means $B=F_i\vee G_j$. Since $F_i$ is left-closed, $x\in F_i$. Similarly, $y\in G_j$ since $G_j$ is right-closed. Hence $\set{x,y}\subseteq F_i\cap G_j$, and we conclude $B= \hcofilter\set{x,y}\subseteq F_i\cap G_j$. In quest of a contradiction, suppose we have an element $z\in(F_i\cap G_j)\setminus B$. First, assume that $\set{x,y,z}$ is an antichain. (This antichain consists of two or three elements, depending on whether $x=y$ or $x\lft y$.) Since $z\in\Betw xy$ would imply $z\in B$, we have $z\lft x$ or $y\lft z$. If $y\lft z$, then $z\in F_i$ implies $z\in F_i\setminus \set y=F_{i-1}$. However, then $y\in F_{i-1}$ since $F_{i-1}$ is left-closed, and this contradicts the definition of $i$. The case $z\lft x$ contradicts the definition of $j$ similarly. Therefore, $\set{x,y,z}$ is not an  antichain.
Since $x\leq z$ and $y\leq z$ are excluded by $z\notin B$, we can assume $z< y$ by left-right symmetry. 
Then $z\in F_i\setminus \set y=F_{i-1}$. Since $F_{i-1}$ is an order-filter, we obtain $y\in F_{i-1}$, which contradicts the definition of $i$. 
\end{proof}

\subsection{Statements on planar, slim, semimodular lattice diagrams}
Let $D$ be a planar lattice diagram.
If $a\leq b\in D$, then the interval $[a,b]$ determines a subdiagram, which is denoted by $[a,b]_D$ or, if there is no danger of confusion, by $[a,b]$. An element of $D$ is a  \emph{narrows} of $D$ if it is comparable with every element of $D$. The set of narrows is denoted by $\Nar D$. Reflecting $D$ to a vertical axis, we obtain its \emph{vertical mirror image} $\refl D$. We need the following statement, which is somewhat stronger than Lemma~\ref{lemmaleftbounddeterm}. 
\begin{lemma}[{\init{G.\ }Cz\'edli and \init{E.\,T.\ }Schmidt  and \cite[Lemma 4.7]{czgschslim2} or, more explicitly, 
\init{G.\ }Cz\'edli and \init{G.\ }Gr\"atzer \cite{czgggltsta}}]
\label{lemMaVflip} Let $D$ and $E$ be finite, slim, semimodular lattice diagrams, and let $\Nar D=\set{0=d_0<d_1<\dots<d_m=1}$ and $\Nar E=\set{0=e_0<e_1<\dots<e_n=1}$. Then $D$ and $E$ determine isomorphic lattices if and only if $m=n$ and, up to similarity, $[d_{i-1},d_i]_D\in\bigset{[e_{i-1},e_i]_E ,\refl{[e_{i-1},e_i]_E } }$ for $i=1,\dots,n$.
\end{lemma}

Next, we recall some well-known facts; see, for example, 
\init{D.\ }Kelly and \init{I.\ }Rival \cite[Proposition 5.2]{kellyrival} and  \init{G.\ }Cz\'edli and \init{G.\ }Gr\"atzer \cite[Exercises 1.5 and 1.5]{czgggltsta}.
The \emph{order dimension} of a poset $P=\tuple{P;\leq}$ is the least $n$ such that the ordering relation ``$\leq$'' is the intersection of $n$ linear (that is, chain) orderings. Equivalently, it is the least $n$ such that $P$ can be order-embedded into the direct product of $n$ chains. A finite lattice has a planar diagram if{f} it is of order-dimension at most 2. Now we are ready to state and prove the following lemma.

\begin{lemma}\label{lemmaalfadefok} If $D$ is a finite, planar, slim, semimodular lattice diagram, then $\alpha(D)$ defined in Definition~\ref{alfaDdef} exists $(\kern-1.5pt$and it is a quasiplanar diagram$)$. 
\end{lemma}

\begin{figure}
\centerline
{\includegraphics[scale=1.0]{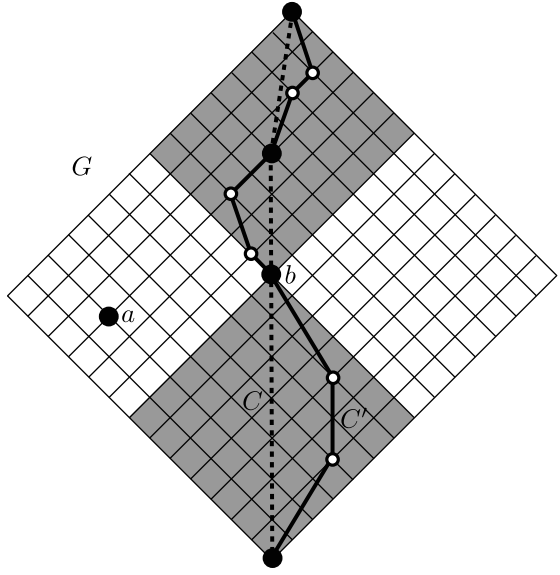}}
\caption{Illustrating the proof of Lemma~\ref{lemmaalfadefok}\label{fig1}}
\end{figure}

\begin{proof} First, we assume that  $\Nar D=\set{0,1}$. By a \emph{grid} we mean a planar diagram of a direct product of two chains such that every edge is of slope $45^\circ$ or $135^\circ$. Let $L$ denote the lattice determined by $D$.
Since $L$ is planar, is has a planar diagram $E$ embedded into a grid $G$, see Figure~\ref{fig1}. The points of $G$ are the intersections of the thin lines, and $E$ consists of the (empty and black-filled) circles and the thick solid lines. The elements of $\Jir L=\Jir E$ are denoted by black-filled circles, and these black-filled circles together with the thick dotted lines
form a diagram of $\Jir L$, which we denote by $P$. Only a part of $E$ and a part of $P$ are depicted. Note that 
\begin{equation}\label{notegymfleTt}
\text{if $u,v\in P$ and $u\parallel v$, then the line through $u$ and $v$ is not vertical.}
\end{equation}

Assume that $a,b\in P$ are incomparable elements, and $a$ is on the left of $b$ in $P$. Pick a 
maximal chain $C$ through $b$ in $P$; it consists of the thick dotted lines. Since $a\parallel b$ also in $E$ and thus in $G$, none of the two  gray-filled closed rectangles can contain $a$. Extend $C$ to a maximal chain $C'$ of $E$. Since the elements of $C'$ are comparable with the elements of $C$  in $G$, we obtain that $C'$ goes in the union of the grey-filled rectangles. Hence $a$ is on the left of $b$ in $E$. Thus we have shown that if $a\lft b$ in $P$, then $a\lft b$ in $E$. This implies that, for $a,b\in P$
\begin{equation}\label{dzWbT}
\text{if $a\lft b$ in $P$, then $a\lft b$ in $G$ and $x(a)<x(b)$,}
\end{equation}
where $x(a)$ and $x(b)$ denotes the first coordinates of $a$ and $b$, respectively. 
Similarly, if $a$ is on the right of $b$ in $P$, then so is in $E$. 

Not all edges (the thick dotted lines) of $P$ are depicted in the figure. If some edge $e$ of $P$ goes through a vertex $v$ of $P$ such that $v$ is not an endpoint of $e$, then we can move $v$ by a very little distance without changing the $\slft$ relation or destroying the validity \eqref{notegymfleTt}. Finally, \eqref{notegymfleTt} allows us to add a zero $\nnul$ and a unit to $P$, and this way we obtain a diagram $Q$.  (Here  \eqref{notegymfleTt} and \eqref{dzWbT} ensure that, if we go high enough, we can find an appropriate position for a new unit, a dually for the new zero.)
Since the grey rectangles above did not depend on the choice of $C$, we conclude that $Q$ is quasiplanar. If $E=D$, up to similarity, 
then we can let $\alpha(D)=Q$. 
Otherwise, by Lemma~\ref{lemMaVflip}, $E=\refl D$ and we can take $\alpha(D)=\refl Q$. This proves the statement for the case $\Nar D=\set{0,1}$.

Second, assume $\Nar D=\set{0=d_0<d_1<\dots<d_m=1}$. The method above gives appropriate $P_i$ for each  $[d_{i-1},d_i]_D$. From these  $P_i$, $i\in\set{1,\dots,m}$, we can easily construct $\alpha(D)$ by putting $P_i$ above $P_{i-1}$ for $i\in\set{1,\dots,m}$, adding a new zero and unit, and adding some edges between $P_{i-1}$ and $P_i$, the new zero and $P_1$, and $P_m$ and the new unit. As before, if a new edge goes through a vertex, we can slightly remove the vertex.
\end{proof}

Now we import two  statements from \init{G.\ }Cz\'edli \cite{czgcircles}.
We say that $y$ is \emph{horizontally between} $x_0$ and $x_1$ if 
$x_0\lft y\lft x_1$ or $x_1\lft y\lft x_0$. Note that $\set{x_0,x_1,y}$ is a 3-element antichain in this case. 

\begin{lemma}[{\init{G.\ }Cz\'edli \cite{czgcircles}}]\label{lMakoztes} 
Let $D$ be a finite, planar lattice diagram, and let $\set{x_0,x_1,y}$ be a $3$-element antichain in  $D$. 
Then the following two statements hold. 
\begin{enumeratei}
\item\label{lMakoztesa} If 
$y$ is horizontally between $x_0$ and $x_1$, then $x_0\wedge x_1  \leq y$.
\item\label{lMakoztesb} If, in addition, $D$ is slim and $x_0\wedge x_1\leq y$, then  $y$ is horizontally between $x_0$ and $x_1$.
\end{enumeratei} 
\end{lemma}

\begin{lemma}[{\init{G.\ }Cz\'edli \cite{czgcircles}}]\label{lMaszdobzF}
If $L$ be a finite  semimodular lattice, $a\in \Mir L$, $b,c\in L$, $a<c$, and $ b\wedge c \leq a$, then $b\leq a$.
\end{lemma}

The following lemma is a particular case of \init{G.\ }Cz\'edli and \init{E.\,T.\ }Schmidt \cite[Lemma 2.2]{czgschtJH}. The leftmost and the rightmost maximal chain of a planar lattice diagram $D$ are the \emph{left boundary chain}, denoted by $\lbound D$, and the \emph{right boundary chain}, denoted by $\rbound D$, respectively.

\begin{lemma}[\cite{czgschtJH}]\label{mxlnckhtlnsK} Let $C_1$ and $C_2$ be maximal chains in a finite, slim, semimodular lattice $L$ such that $\Jir L\subseteq C_1\cup C_2$. Then $L$ has a planar diagram $D$ such that $C_1=\lbound D$ and $C_2=\rbound D$. Furthermore, this diagram is unique $($up to similarity$)$.
\end{lemma}

\subsection{Join and meet representations in slim, semimodular lattices}
\begin{definition}\label{sprtdeF}
For $x$ in a planar lattice diagram $ D$, the largest element of $\ideal x\cap \lbound D$ and that of $\ideal x\cap \rbound D$ are the \emph{left support} of $x$, denoted by $\lsp x$, and the 
\emph{right support} of $x$, denoted by $\rsp x$, respectively. 
\end{definition}

It follows from the definition of slimness that 
\begin{equation}\label{supplrgx}
x=\lsp x\vee \rsp x\text{, for all }x\in D,
\end{equation}
provided $D$ is a planar, slim lattice diagram.

\begin{lemma}\label{lemmalsprsplambda}
For $x\parallel y$ in a planar, slim, semimodular lattice diagram $D$, we have $x\lft y$ if{f} $\,\lsp x>\lsp y$ and $\rsp x<\rsp y$. Furthermore, $x\leq y$ if{f} $\,\lsp x\leq \lsp y$ and $\rsp x\leq\rsp y$
\end{lemma}

\begin{proof} Assume $x\lft y$. If $\lsp x=\lsp y$, then $\rsp x\nonparallel \rsp y$ since $\rbound D$ is a chain, and \eqref{supplrgx} gives $x\nonparallel y$. Hence, $\lsp x\neq \lsp y$ and $\rsp x\neq \rsp y$. 

Assume $x\lft y$. Striving for a contradiction, suppose $\lsp x<\lsp y$.
By the definition of $\lsp x$, we have $\lsp y\not\leq x$. On the other hand, $x\parallel y\geq \lsp y$ implies $\lsp y\not\geq x$. That is, $\lsp y\parallel x$. Since $x$ is on the right of $\lbound D$, Lemma~\ref{leftrightlemma} yields $\lsp y\lft x$. Take a  
a maximal chain $C$ through $\set{\lsp x,x}$. Lemma~\ref{leftrightlemma}, $\lsp y\lft x$, and $x\lft y$ yield that $\lsp y$ is on the right of $C$ and $y$ is on the left of $C$. Hence, by 
Lemma~\ref{krchainlemma}, there exists a $c\in C$ such that $\lsp y\leq c\leq y$. Belonging to the same chain, $c$ and $x$ are comparable. Since $x\not\leq y$, we conclude $c<x$. However, then $\lsp x<\lsp y\leq c<x$ and $\lsp y\in \lbound D$ contradict the definition of $\lsp x$. 

Therefore, $x\lft y$ implies $\lsp x>\lsp y$. By left-right duality, it also implies $\rsp x<\rsp y$. This proves the ``only if'' part of the lemma.
To prove the ``if'' part, assume $\lsp x>\lsp y$ and $\rsp x<\rsp y$. Clearly, $x\parallel y$. We cannot have $y\lft x$ since it would contradict the ``only if'' part. Thus $x\lft y$.

Finally, the second statement of the lemma is obvious.
\end{proof}

As a counterpart of Definition~\ref{sprtdeF}, we present the following concept.

\begin{definition}Let $D$ be a finite, slim, semimodular lattice diagram, and let $b\in D\setminus\set 1$. The \emph{left dual support} and the \emph{right dual support} of $b$, denoted by $\lds b$ and $\rds b$,  are the leftmost and the rightmost element of the antichain
$\Min{(\filter b\cap\Mir D)}$, respectively. 
\end{definition}

A meet $x_1\wedge\dots\wedge x_n$ in a lattice is \emph{irredundant} if 
\[x_1\wedge\dots \wedge x_{i-1}\wedge x_{i+1}\wedge\dots \wedge x_n\neq x_1\wedge\dots\wedge x_n\]
for $i=1,\dots,n$. 

\begin{lemma}\label{mdiWrT} Let $D$ be a finite, slim, semimodular lattice diagram, and let $b\in D\setminus\set 1$. Then $b=\lds b\wedge \rds b$. Furthermore, if 
$X\subseteq \Mir L$ such that $b=\bigwedge X$ is an irredundant meet representation of $b$, then $X=\set{\lds b,\rds b}$.
\end{lemma}

\begin{proof} Obviously, $b=\bigwedge \Min{(\filter b\cap\Mir D)}$.  Lemma~\ref{lMakoztes}\eqref{lMakoztesa} implies $b=\lds b\wedge \rds b$. 
If $\lds b\neq \rds b$, then $\lds b\parallel \rds b$ and $b=\lds b\wedge \rds b$ is an irredundant-meet representation. Hence, with the notation  $Y=\set{\lds b,\rds b}$, $b=\bigwedge Y$ is an irredundant meet-representation, even if $\lds b = \rds b$. Since slim semimodular lattices are join-distributive, 
see \init{G.\ }Cz\'edli, {L.\ }Ozsv\'art, and  \init{B.\ }Udvari \cite[Corollary 2.2]{czgolub}, and the irredundant meet-representation in join-distributive are unique by \init{R.\,P.\ }Dilworth \cite{r:dilworth40}, the rest of the lemma follows.  
\end{proof}

As a counterpart of Lemma~\ref{lemmalsprsplambda}, we have the following.


\begin{lemma}\label{lMadsplrdu}
Let $x$ and $y$ be elements of a planar, slim, semimodular lattice diagram $D$. Then the following two assertions hold.
\begin{enumeratei}
\item\label{lMadsplrdua} $x\leq y$ if{f} $\,\lds x \leqlft \lds y$ and $\rds x\leqrght \rds y$;
\item\label{lMadsplrdub} $x\lft y$ if{f} $\,\lds x \leslft \lds y$ and $\rds x\bigglft \rds y$.
\end{enumeratei}
\end{lemma}

\begin{proof} We shall use the identity $b=\lds b\wedge \rds b$ of Lemma~\ref{mdiWrT} without further reference. 
Assume $x\leq y$. Then $\lds x\wedge \rds x\leq y\leq \lds y$. If $\lds x\parallel \lds y$, then $\lds x\lft \rds x$, and  $\lds x\wedge \rds x\leq \lds y$ implies $\lds x\lft \lds  y$ by Lemma~\ref{lMakoztes}\eqref{lMakoztesb}. If $\lds y<\lds x$, then Lemma~\ref{lMaszdobzF}  with 
\[\tuple{a,b,c}=\tuple{\lds y, \rds x, \lds x}\]
implies $\rds x\leq \lds y<\lds x$, which is a contradiction since $\lds x\parallel \rds x$ or $\lds x = \rds x$. Thus if 
$\lds x\nonparallel \lds y$, then $\lds x\leq \lds y$. Hence,  $\lds x \leqlft \lds y$; $\rds x\leqrght \rds y$ follows  by left-right duality. This proves the ``only if'' part of \eqref{lMadsplrdua}.

To prove the ``if'' part, assume $\lds x \leqlft \lds y$ and $\rds x\leqrght \rds y$. If $\lds x \lft \lds y$ and $\rds x\rght \rds y$, then $\lds x\lft\lds y \elft \rds y\lft \rds x$ and Lemma~\ref{lMakoztes}\eqref{lMakoztesa}  
imply $x=\lds x\wedge \rds y \leq \lds y$ and $x\leq \rds y$, and we obtain $x\leq y$. If $\lds x \leq \lds y$ and $\rds x\leq \rds y$, then $x\leq y$ trivially follows. There are two more cases; we only deal with one of them, because the other one will follow by left-right duality. Assume $\lds x \leq \lds y$ and $\rds x\rght \rds y$. If $\lds x\nonparallel \rds y$, then $\lds x\leq \rds y$ and thus $x\leq \lds x\leq \lds y\wedge \rds y=y$, because $\lds x > \rds y$ would imply $\lds y\geq x>\rds y$, a contradiction.  If $\lds x\parallel \rds y$, then  $\lds y\lft \rds y$ and  Lemma~\ref{leftrightlemma}, applied to a maximal chain through $\set{\lds x,\lds y}$, yield $\lds x\lft \rds y\lft \rds x$, Lemma~\ref{lMakoztes}\eqref{lMakoztesb} gives $\lds x\wedge\rds x\leq \rds y$, and we conclude  
$x= \lds x\wedge\rds x\leq \lds y\wedge\rds y=y$ again. This proves \eqref{lMadsplrdua}. 

To prove the ``only if'' part of \eqref{lMadsplrdub}, assume $x\lft y$. Striving for a contradiction, suppose $\lds x=\lds y$. We have $\rds x\parallel \rds y$ since otherwise $x=\lds x\wedge \rds x$ and $y=\lds x\wedge \rds y$ would be comparable. If  $\rds x\lft \rds y$, then $\lds y=\lds x \lft \rds x\lft \rds y$ and Lemma~\ref{lMakoztes}\eqref{lMakoztesa} imply $x\leq \rds x\leq \lds y\wedge\rds y=y$, a contradiction. Similarly, if  $\rds y\lft \rds x$, then $\lds x=\lds y \lft \rds y\lft \rds x$ and Lemma~\ref{lMakoztes}\eqref{lMakoztesa} imply $y\leq \rds y\leq \lds x\wedge\rds x=x$, a contradiction again. This proves $\lds x\neq \lds y$. 

Next, aiming at contradiction again, suppose $\lds x > \lds y$. Extend the chain $\set{y\leq \lds y<\lds x}$ to a maximal chain $C_1$. Since $x\lft y$ and $\lds x\elft \rds x$, we obtain that  $x$ is on the left of $C_1$ and $\rds x$ is on the right of $C_1$.  Lemma~\ref{krchainlemma} yield an element $z\in C_1$ such that $x\leq z\leq \rds x$. We have $x<z$ since $x\notin C_1$, and $\lds x\not<z$ since $\lds x\not<\rds x$. Belonging to the same chain, $z$ and $\lds x$ are comparable, and we obtain $z\leq \lds x$. This gives $x<z\leq \lds x\wedge \rds x=x$, a contradiction. Therefore, $\lds x\not>\lds y$.

For the sake of the next contradiction, suppose $\lds y\lft \lds x$. Extend $\set{x,\lds x}$ to a maximal chain $C_2$. Since $\lds y\lft \lds x$, $\lds y$ is on the left of $C_2$, while $x\lft y$ yields that  $y$ is on the right of $C_2$. Hence Lemma~\ref{krchainlemma}  applies, and we obtain an element $z\in C_2$ such that $y\leq z\leq \lds y$. 
Since  $z\nonparallel x$, as both belong to $C_2$,  and $x\not>y$, we have $x<z$, and thus $x<\lds y$. Now the set $\filter x\cap\ideal{\lds y}\cap\Mir D$ is nonempty since it contains $\lds y$. Let $t$ be a minimal element of this set. Clearly, $t$ belongs to the antichain $\Min{(\filter x\cap\Mir D)}$. Since  $\lds x$ is the leftmost element of this antichain, we have $\lds x\elft t$. We cannot have $\lds x=t$, because otherwise $\lds x=t\leq \lds y$ would contradict $\lds y\lft \lds x$. Hence $\lds x \lft t$. 
Now extend $\set{t,\lds y}$ to a maximal chain $C_3$. Then $\lds x$ is on the left of $C_3$ since $\lds x\lft t$, and $\lds x$ is also on the right of $C_3$ since $\lds y\lft \lds x$. Therefore, $\lds x\in C_3$ and thus $\lds x\nonparallel \lds y$, which contradicts $\lds y\lft \lds x$. This proves that $\lds y\lft \lds x$ is impossible. 

Now, that we have excluded all other possibilities, we conclude that $x\lft y$ implies $\lds x \leslft \lds y$. By left-right duality, it also implies  $\rds x\bigglft \rds y$. This proves the ``only if'' part of \eqref{lMadsplrdub}. 
Finally, to prove the ````if'' part of \eqref{lMadsplrdub}, assume $\lds x \leslft \lds y$ and $\rds x\bigglft \rds y$. Part \eqref{lMadsplrdua} excludes $x\nonparallel y$, and the ``only if'' part of \eqref{lMadsplrdub} excludes $y\lft x$. Hence, $x\lft y$.
\end{proof}


\subsection{Further auxiliary statements}
\begin{lemma}\label{MIRdescR}  If $Q$ is a quasiplanar diagram, then 
\[\Mir{\tuple{\Hco Q;\dleq}}= \set{\filter x: x\in Q\setminus\set{0,1}}\text.
\]
\end{lemma}

\begin{proof} Let $F\in \Hco Q$. By Lemma~\ref{phipilemMa}, $F$ is of the form $F=\hcofilter\set{x,y}$, where $x=\lmost F$, $y=\rmost F$, and   $x\elft y \in \plu Q$. First, assume  $x=y$.  Then $F=\filter x$ by Lemma~\ref{sdGT}\eqref{sdGTa}. Clearly, $F\setminus \set x\in \Hco Q$, and it is the unique lower cover of $F$ with respect to set inclusion. Hence, $F\setminus \set x$ is the unique upper cover of $F$ in  the lattice $\tuple{\Hco Q;\dleq}$. That is, $F\in \Mir{\tuple{\Hco Q;\dleq}}$, proving the ``$\supseteq$'' part of the lemma.

Next, assume $x\neq y$. Obviously, 
$F\neq \filter x$ and $F\neq \filter y$. By Lemma~\ref{sdGT}\eqref{sdGTa}, $\filter x,\filter y\in \Hco Q$. Clearly, 
$F=\filter x\vee \filter y$ in the dual lattice $\tuple{\Hco Q;\subseteq}$. Thus $F=\filter x\wedge \filter y$ in 
$\tuple{\Hco Q;\dleq}$, and $F\notin \Mir{\tuple{\Hco Q;\dleq}}$. This proves $\not\supset$.
\end{proof}

\begin{lemma}\label{uAlftzm} Let $Q$ be a quasiplanar diagram, and let $x,y\in Q$. Then $x\lft y$ in $Q$ if{f}  $\filter x\lft \filter y$ in $\beta_2(Q)$.
\end{lemma}

\begin{proof} To prove the ``only if'' part, assume $x\lft y$, and let $n$ be the smallest subscript such that $y\in F_n$. 
Note that $y\in F_n$ if{f} $\filter y\subseteq F_n$ if{f} $F_n\dleq \filter y$. Note also that $n>k$ if{f} $F_n\dleq F_k$. Therefore, $F_n=\lsp{\filter y}$. 
Also, if $m$ is the smallest subscript such that $x\in F_m$, then $F_m=\lsp{\filter x}$. Since $F_n$ is left-closed, 
$x\in F_n$, which implies $m\leq n$. In fact, $m<n$ since $x\neq y$ yields $m\neq n$. Thus $\lsp{\filter x} = F_m \dsgeq F_n = \lsp{\filter y}$. Left-right duality yields $\rsp{\filter x} \dsleq  \rsp{\filter y}$. Therefore, since $\lbound{\beta_2(Q))=\vec F}$ and $\rbound{\beta_2(Q))=\vec G}$ by 
Definition~\ref{defbetaket}, Lemma~\ref{lMvecFvecBjoIn}, and Lemma~\ref{mxlnckhtlnsK}, we can apply Lemma~\ref{lemmalsprsplambda} to obtain $\filter x\lft \filter y$ in $\beta_2(Q)$.  This proves the ``only if'' part.  

Conversely, assume $\filter x\lft \filter y$ in $\beta_2(Q)$. Then, in particular, $\filter x\parallel \filter y$. 
Clearly, 
$\text{$u\leq v$ in $Q$ if{f} $\filter u\supseteq \filter v$ if{f}  $\filter u\dleq \filter v$ in $\tuple{\Hco Q;\dleq}$.}
$
In particular, $u\parallel v$ in $Q$ if{f} $\filter u\parallel \filter v$ in $\tuple{\Hco Q;\dleq}$. 
This yields $x\parallel y$. Hence $x\lft y$ or $y\lft x$ in $Q$. Since $y\lft x$ would give a contradiction by the ``only if'' part, we obtain  $x\lft y$.
\end{proof}

\begin{lemma}\label{betaonetWo} If $Q$ is a quasiplanar diagram, then the planar diagrams  $\beta_1(Q)$ and $\beta_2(Q)$ are the same, up to similarity.
\end{lemma}

\begin{proof} First, as a preparation to use Lemma~\ref{lMadsplrdu}, 
we show that if $X\in \Hco Q$, then 
\begin{equation}\label{nBmYW}
\lds X=\filter{\lmost X}\,\text{ and }\,\rds X=\filter{\rmost X}\text.
\end{equation}
It follows from Lemma~\ref{sdGT} that $X=\filter{\MinBetw {\lmost X}{\rmost X} }$.
We know from Lemma~\ref{MIRdescR} that the meet-irreducible elements of $\beta_2(Q)$ are exactly the $\filter x$, $x\in Q\setminus\set{0,1}$. We have to consider the minimal ones above $X$, with respect to ``$\dleq$''. That is, the maximal ones below $X$, with respect to set inclusion. Clearly, they are the members of $A=\set{\filter x: x\in\MinBetw {\lmost X}{\rmost X}}$. By definition,  $\lds X$ is the leftmost member of $A$ with respect to $\lft$ defined in $\beta_2(Q)$.  Hence, by Lemma~\ref{uAlftzm}, $\lds X=\filter{\lmost X}$. The rest of \eqref{nBmYW} follows similarly.

Next, consider the order-isomorphism $\filttoep\colon \Hco Q\to \Elig (Q)$, defined by  $F\mapsto \pair{\lmost F}{\rmost F}$ in Lemma~\ref{phipilemMa}. To show that $\filttoep$ preserves the relation $\lft$, assume that $X_1,X_2\in \Hco Q$ and $X_1\lft X_2$. Let $x_i=\lmost{X_i}$ and $y_i=\rmost{X_i}=y_i$.  By Lemma~\ref{sdGT}, we have $X_i=\hcofilter{\set{x_i,y_i}}$ for $i\in\set{1,2}$. With reference to the notation introduced in Definition~\ref{defbetaket}, we claim that
\begin{equation}\label{dTMgHpp}
\lsp{X_i}=F_{n_i}\iff n_i=\min\set{j: y_i\in F_j}\text.
\end{equation} 
To see this, we can argue as follows: $\lsp{X_i}=F_k$ $\iff$ $F_k\dleq X_i$ and $F_k$ is maximal with respect to $\dleq$ 
 $\iff$ $F_k\supseteq X_i$ and $F_k$ is minimal with respect to set inclusion  $\iff$ $y_i\in F_k$ and $k$ is minimal; in the last step we used that $F_k$ is left-closed by Lemma~\ref{lMvecFvecBjoIn} and $x_i\elft y_i$, and thus  $y_i\in F_k$ implies $x_i\in F_k$. This proves \eqref{dTMgHpp}. 

From Lemma~\ref{lemmalsprsplambda}, we obtain $\lsp{X_1}\dsgeq \lsp{X_2}$. This and \eqref{dTMgHpp} yield that 
$F_{n_1}=\lsp{X_1}\subset \lsp{X_2}= F_{n_2}$, $y_1\in F_{n_1}$, $y_2\in F_{n_2}$,  and $y_2\notin F_{n_1}$ since we have $n_1<n_2$ by $F_{n_1}\subset F_{n_2}$. Since $n_1\neq n_2$, we have $y_1\neq y_2$. Hence, either $y_1\bigglft y_2$, or  $y_2\bigglft y_1$. However, if we had $y_2\bigglft y_1$, then we would obtain that $y_2$ belongs to $F_{n_1}$ since $y_1\in F_{n_1}$ and $F_{n_1}$ is left-closed by Lemma~\ref{lMvecFvecBjoIn}, and this would be a contradiction. Consequently, $y_1\bigglft y_2$. The left-right dual of the argument above gives $x_1\leslft x_2$. Hence, by \eqref{defbetaegyb}, we obtain 
$\filttoep(X_1)=\pair{x_1}{y_1} \lft \pair{x_2}{y_2}= \filttoep(X_2)$. This means that $\filttoep$ preserves $\lft$.  

Finally, to show that $\eptofilt=\filttoep^{-1}$ preserves $\lft$, assume that $\filttoep(X_1) \lft \filttoep(X_2)$. Then we have $X_1\parallel X_2$ since $\filttoep$ is an order-isomorphism by Lemma~\ref{phipilemMa}. Thus either $X_1\lft X_2$, or $X_2\lft X_1$. However, $X_2\lft X_1$ would imply the contradiction $\filttoep(X_2) \lft  \filttoep(X_1)$ since $\filttoep$ preserves $\lft$. Hence, $X_1\lft X_2$. 
\end{proof}

\subsection{The end of the proof}
Armed with the auxiliary statements presented so far, now we are in the position to accomplish our goal.

\begin{proof}[Proof of Theorem~\ref{thmmain}]
By Lemma~\ref{betaonetWo}, $\beta_1(Q)$ equals $\beta_2(Q)$, up to similarity. Hence, in what follows, no matter if  $p$ is 1 or 2, we can use any of $\beta_1$ and $\beta_2$. 

Part \eqref{thmmaina} is Lemma~\ref{lemmaalfadefok}, while Part \eqref{thmmainb} follows from Lemmas~\ref{nhVbztZTsV},  \ref{lMvecFvecBjoIn}, and \ref{mxlnckhtlnsK}.

To prove Part  \eqref{thmmainc}, let $D$ be a  finite, slim, semimodular lattice diagram, and let $Q=\alpha(D)$. Define a map $\alpha\colon D\to 
\beta_1(Q)$ by $x\mapsto \pair{\lds x}{\rds x}\in \Elig(Q)$.  (Here, for technical reasons, we extend the definition of $\lds x$ and $\rds x$ by letting $\lds 1=\rds 1=1$; this will cause no problem and makes the definition of $\alpha$ meaningful.)   Since $x=\lds x\wedge \rds x$ by Lemma~\ref{mdiWrT}, $\alpha$ is injective. Assume $\pair yz\in \Elig(Q)$ such that $y\neq z$, and define $x$ by $x=y\wedge z$. This is an irredundant meet representation since $y\parallel z$. By the uniqueness part of  Lemma~\ref{mdiWrT} and $y\lft z$, we obtain $\pair yz=\pair{\lds x}{\rds x}=\alpha(x)$. Hence, $\alpha$ is surjective. Finally, comparing 
Lemma~\ref{lMadsplrdu}\eqref{lMadsplrdua}
to \eqref{defbetaegya} and Lemma~\ref{lMadsplrdu}\eqref{lMadsplrdub} to \eqref{defbetaegyb}, we conclude that $\alpha$ is similarity map. This proves Part  \eqref{thmmainc}.

To prove Part  \eqref{thmmaind}, let $Q$ be a quasiplanar diagram. Combining Lemmas~\ref{MIRdescR} and  \ref{betaonetWo}, we conclude $\Mir{(\beta_1(Q))}=\bigset{\pair xx: x\in Q\setminus\set{0,1}}$. To form $\alpha\bigl(\beta_1(Q))\bigr)$, we have to add a bottom and a top to $\Mir{(\beta_1(Q))}$; denote them by $\pair 00$ and $\pair 11$, respectively. 
Then we have $\alpha\bigl(\beta_1(Q))\bigr)= \set{\pair xx: x\in Q}$.
We claim that $\gamma\colon Q\to \alpha\bigl(\beta_1(Q))\bigr)$, defined by $x\mapsto \pair xx$, is a similarity map. Obviously, $\gamma$ is a bijection.
Since the position of a top or bottom element in a diagram is unique up to similarity, it suffices to deal with the elements of $Q\setminus\set{0,1}$.
Assume $x,y\in Q\setminus\set{0,1}$. Based on  \eqref{defbetaegya}, we have
\[ \pair xx\leq \pair yy \iff x\leqlft y \text{ and }y\geqlft x \iff x\leq y,
\]
which shows that $\gamma$ is an order-isomorphism.
Based on  \eqref{defbetaegyb}, we obtain
\[ \pair xx\leq \pair yy \iff 
x \leslft y\text{ and }x\bigglft y \iff x\lft y\text.
\]
Therefore, $\gamma$ is a similarity map, completing the proof of Part  \eqref{thmmaind}.
\end{proof}

\section{Comments and examples}\label{commentsection}
One may ask which finite, bounded posets have quasiplanar diagrams. 

\begin{proposition} A finite, bounded partially ordered set $P$ has a quasiplanar diagram if{f} its order dimension is at most two.
\end{proposition}

\begin{proof}   Assume that $P$ is quasiplanar. By Theorem~\ref{thmmain}, $P$ can be order-embedded into a finite, slim, semimodular lattice $L$. Since $L$ has a planar diagram by Lemma~\ref{mxlnckhtlnsK}, cited from \init{G.\ }Cz\'edli and \init{E.\,T.\ }Schmidt \cite{czgschtJH},  it is of order-dimension at most two. Thus $P$ is of order-dimension at most two.

Next, assume that $P$ is of order-dimension at most two. Then $P$ has a diagram that is a subdiagram of a grid $G$, like in Figure~\ref{fig1}. Let $Q$ be the diagram of $P$ that is obtained from $G$ by deleting superfluous grid points and connecting covering elements of $P$ by straight line segments. For the sake of contradiction, suppose $Q$ is not quasiplanar. Then there are $x\nonparallel y\in Q$ and maximal chains $C_1$ and $C_2$ of $Q$ such that $y\in C_1\cap C_2$, $x$ is strictly on the right of $C_1$, and it is strictly on the left of $C_2$.  
Let $E_1$ and $E_2$ be the leftmost and the rightmost maximal chains of $G$ that extend $C_1$ and $C_2$, respectively. Then $x$ is on the right of $E_1$, it is  on the left of $E_2$ in $G$, and $x\parallel y$ in $G$. But this is a contradiction since $G$ is quasiplanar, in fact, it is planar.
\end{proof}

\begin{figure}
\centerline
{\includegraphics[scale=1.0]{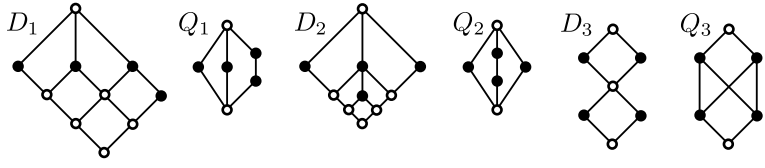}}
\caption{$Q_1$ is order-isomorphic to $Q_2$ but $D_1\neq D_2$    \label{fig2}}
\end{figure}

We conclude the paper with some examples. In Figure~\ref{fig2}, $Q_i=\alpha(D_i)$, and the meet irreducible elements are black-filled. The figure explains why we deal with diagrams rather than lattices and posets: order-isomorphic quasiplanar diagrams can determine non-isomorphic lattices. 
 Also, $D_3$ is the smallest slim, semimodular lattice diagram such that $Q_3=\alpha(D_3)$ is not planar, and there  is no planar diagram order-isomorphic to $\alpha(D_3)$. Finally,  Figure~\ref{fig3} illustrates that 
Lemma~\ref{lemmaalfadefok} is not so obvious as it may look. In the figure, $D_4$ and $D_5$ are equal, up to similarity. 
For $i\in \set{4,5}$, $Q_i$ is obtained from 
$D_i$ by omitting vertices and connecting the remaining ones, without changing their position. We have $Q_4=\alpha(D_4)$. However,
$Q_5\neq\alpha(D_5)$, because $Q_5$ is not a quasiplanar diagram since $c\parallel a$, $c$ is on the left of the chain $\set{0,a,f,1}$ through $a$, but $c$ is on the right of the chain $\set{0,a,d,1}$ through~$a$.

\begin{figure}
\centerline
{\includegraphics[scale=1.0]{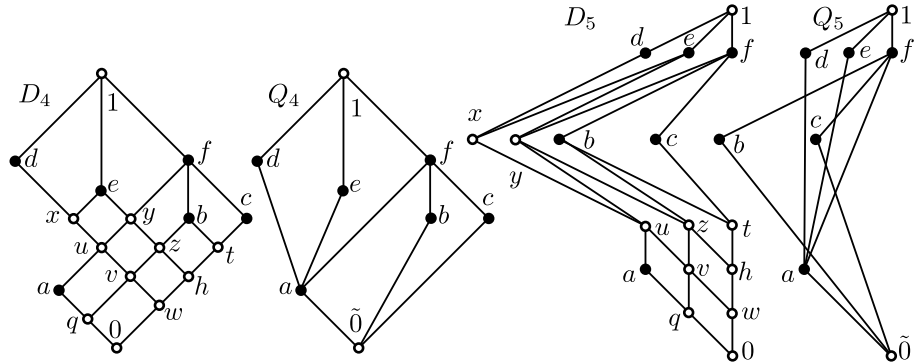}}
\caption{$D_4=D_5$ and $Q_4=\alpha(D_4)$, but $Q_5\neq\alpha(D_5)$ \label{fig3}}
\end{figure}

\end{document}